\documentclass[12pt,a4paper]{article}
\usepackage[dvips]{color}
\usepackage{indentfirst,latexsym,bm}
\usepackage{graphicx}
\usepackage{amsmath,amsthm}
\usepackage{amssymb}
\usepackage{multicol}
\usepackage{ulem}
\usepackage{amsfonts}
\usepackage{mathrsfs}
 \newtheorem{thm}{Theorem}[section]
 \newtheorem{cor}[thm]{Corollary}
 
 \newtheorem{prop}[thm]{Proposition}
 
 \theoremstyle{definition}
 \newtheorem{dfn}[thm]{Definition}

 \newtheorem{rem}{Remark}

  \DeclareMathAlphabet{\mathsfsl}{OT1}{cmss}{m}{sl}

 \newcommand{\Rnum}{\mathbb{R}}
 \newcommand{\Cnum}{\mathbb{C}}
 
 \newcommand{\Nnum}{\mathbb{N}}
 \newcommand{\mi}{\mathrm{i}}
 \newcommand{\dif}{\mathrm{d}}

 \newcommand{\abs}[1]{\left\vert#1\right\vert}
 \newcommand{\set}[1]{\left\{#1\right\}}
 \newcommand{\norm}[1]{\left\Vert#1\right\Vert}
 \newcommand{\innp}[1]{\langle {#1}\rangle}

\allowdisplaybreaks

\title{On a nonsymmetric Ornstein-Uhlenbeck semigroup and its generator}
\author{\rm\small
\noindent CHEN Yong\\
\noindent \footnotesize School of Mathematics and Computing Science, Hunan
University of Science and Technology,\\
\noindent \footnotesize Xiangtan, Hunan, 411201,
P.R.China. chenyong77@gmail.com\\
}
\date{}
\begin{document}
\maketitle
\maketitle \noindent {\bf Abstract } \\
 If we add a simple rotation term to both the Ornstein-Uhlenbeck semigroup and definition of the H-derivative, then analogue to the classical Malliavin calculus on the real Wiener space [I. Shigekawa, Stochastic analysis, 2004], we get a normal but nonsymmetric Ornstein-Uhlenbeck operator $L$ on the complex Wiener space. The eigenfunctions of the operator $L$ are given. In addition, the hypercontractivity for the nonsymmetric Ornstein-Uhlenbeck semigroup is shown.
\maketitle
\section{Introduction}
In \cite{cl}, the following stochastic differential equation is considered:
\begin{equation}\label{cp}
\left\{
      \begin{array}{ll}
      \dif Z_t=-\alpha Z_t\dif t+ \sqrt{2\sigma^2}\dif \zeta_t,\quad t\ge 0,  \\
      Z_0=z_0\in \Cnum^1,
      \end{array}
\right.
\end{equation}
where $Z_t=X_1(t)+\mi X_2(t)$, $\alpha=ae^{\mi \theta}= r+\mi \Omega$ with $a>0,\,\theta\in (-\frac{\pi}{2},\,\frac{\pi}{2})$, and $\zeta_t=B_1(t)+\mi B_2(t)$ is a complex Brownian motion. Clearly, when $\Omega\neq 0$, the generator of the process is a 2-dimensional not symmetric but normal Ornstein-Uhlenbeck (OU) operator
\begin{align}\label{ou.op}
A & =\sigma^2(\frac{\partial^2}{\partial x^2}+\frac{\partial^2}{\partial y^2})+(-rx+\Omega y)\frac{\partial}{\partial x}-(\Omega x+ry)\frac{\partial}{\partial y}
\\
&=4\sigma^2 \frac{\partial^2}{\partial z\partial \bar{z}}-\alpha z \frac{\partial}{\partial z}-\bar{\alpha}\bar{z} \frac{\partial}{\partial \bar{z}},
\end{align}
where we denote by
$\frac{\partial f}{\partial z}=\frac12 (\frac{\partial f}{\partial x}-\mi \frac{\partial f}{\partial y}), \frac{\partial f}{\partial \bar{z}}=\frac12 (\frac{\partial f}{\partial x}+\mi \frac{\partial f}{\partial y})$
the formal derivative of $f$ at point $z=x+\mi y$. Note that $\Im(\alpha)\neq 0$ in Eq.(\ref{cp}) is the key point for the nonsymmetric property. The eigenfunctions of $A$ are the so called complex Hermite polynomials\footnote{It is called the Hermite-Laguerre-It$\hat{\rm{o}}$ polynomials in \cite{cl}.} which can be generated iteratively by the complex creation operator acting on the constant $1$. Let $B=\begin{bmatrix} -r & \Omega\\ -\Omega & -r  \end{bmatrix}$ and $B_0=\begin{bmatrix} \cos \Omega t & \sin \Omega t\\ -\sin \Omega t & \cos \Omega t  \end{bmatrix}$.
Then $e^{tB}=e^{-rt}B_0(t) $ and the associated OU semigroup of $ A$ is
\begin{align}
   P_t \varphi (z_0)&=\int_{\Rnum^2}\,\varphi(e^{-rt}B_0(t) z_0+\sqrt{1-e^{-2rt}}z)\, \mu(\dif z)\nonumber\\
   &=\int_{\Cnum}\,\varphi(e^{-\alpha t} z_0+\sqrt{1-e^{-2rt}}z)\, \mu(\dif z),\label{compp}
\end{align}
where the stationary distribution is
$\dif\mu = \frac{r}{2\pi\sigma^2 }\exp\set{-\frac{r(x^2+y^2)}{2\sigma^2}}\dif x\dif y$ and we write a function $\varphi(x,y)$ of the two real variables $x$ and $y$ as the function $\varphi(z)$ of the complex argument $x+\mi y$ (i.e.,
we use the complex representation of $\Rnum^2$ in (\ref{compp})). For simplicity, we can choose that $a=1$ and $r=\sigma^2=\cos \theta$ then (\ref{compp}) becomes \begin{equation}
   \int_{\Cnum}\,\varphi(e^{- (\cos \theta+\mi \sin \theta)t} z_0+\sqrt{1-e^{-2t\cos \theta}}z)\, \mu(\dif z).
\end{equation}

If we let $z_0,z$ be in the infinite dimensional space $(C_0([0,T]\to \Cnum^1)$, we can define the nonsymmetric OU semigroup on $(C_0([0,T]\to \Cnum^1)$ (see Definition~\ref{ddf_ou}). This idea is similar to the symmetric case \cite{shg}. This is the topic of Section~\ref{ss2}.

The topic of Section~\ref{ss3} is how to obtain a
concrete expression of the generator $L$ of the above OU semigroup with rotation. We extend the definition of the Gateaux derivative and the H-derivative to the function $F:\,B\to \Cnum$ 
and consider the derivative of the function $F(x+e^{\mi \theta}t y)$ with $t\in \Rnum$ and $\theta\in (-\frac{\pi}{2},\frac{\pi}{2})$ at $t=0$ (i.e., here the rotation term in the derivative corresponds to the one in the above OU semigroup). Furthermore, since we consider complex-value functions, we need the conjugate-linear functional. This idea also comes from the symmetric case\cite{shg}.

In Section~\ref{ss4}, we recall the It$\hat{\rm{o}}$-Wiener chaos decomposition and give all the eigenfunctions of the generator $L$. In addition, we show the hypercontractivity for the above OU semigroup.
\section{The nonsymmetric OU semigroup }\label{ss2}

By the complex representation of $\Rnum^2$, the planar Brownian motion $(B^1, B^2)$ will be written $B=B^1+\mi B^2$.
Let $H_1$ be the 1-dimensional Cameron-Martin space \cite{shg}, and denote $H$ the complex Hilbert space $  H=H_1+\mi H_1$
with the natural inner product
\begin{equation}
   (h,k)_{H}=\int_0^T\,\dot{h}(s)\overline{\dot{k}(s)}\dif s.
\end{equation}
Clearly, one can choose a c.o.n.s of $H$ to be $\set{\frac{\varphi_m}{\sqrt{2}},\frac{\bar{\varphi}_m}{\sqrt2}:\,m=1,2,\dots}$.

We look the 2-dimensional Wiener space as a complex Wiener space $(C_0([0,T]\to \Cnum^1),\,\mu)$. The characteristic function of $\mu$ is
\begin{equation}\label{character func}
  \int_{B} \exp\set{\sqrt{-1}\innp{\omega, \varphi}}\,\dif \mu(\omega)=\exp\set{-\frac12 |\varphi|^2_{H^*} },\quad \forall \varphi\in B^*.
\end{equation}
\begin{dfn}\label{ddf_ou}
Let the above notation prevail. We define transition probability on $B$ as follows.
For $\theta\in(-\frac{\pi}{2},\frac{\pi}{2}),\, t\ge 0,\Omega \in \Rnum, \,x\in B,\,A\in \mathcal{B}(B)$ (the Borel $\sigma$-field generated by all open sets),
\begin{equation}\label{semigp}
   P_t(x,\,A)=\int_B\,1_A(e^{-(\cos\theta+\mi \sin\theta)t}x+\sqrt{1-e^{-2t\cos \theta}}y)\,\mu(\dif y).
\end{equation}
\end{dfn}
The following property about the measure $\mu$ is well known.
\begin{prop}\label{rotat}
   For any $a\in \Rnum$, the induced measure of $\mu$ under the mapping $x\mapsto e^{\mi a}x$ is identical to $\mu$, that is to say, $\mu$ is rotation invariant. In addition, for any $t\ge 0$, denote the induced measure of $\mu$ under the mapping $x\mapsto \sqrt{t}x$ by $\mu_t$, then $\mu_t\star \mu_{s}=\mu_{t+s}$ ($\ast$ is the convolution operator).
\end{prop}
An argument similar to the one used in the real case \cite[p21]{shg} shows that $P_t(x,\,A)$ satisfies the Chapman-Kolmogorov equation.
\begin{align*}
   & \int_{B} P_t(x,\,\dif y) P_s(y,\,A)\\
   &= \int_B P_s(e^{-(\cos\theta+\mi \sin\theta)t}x+\sqrt{1-e^{-2t\cos \theta}}y,\,A)\mu(\dif y)\\
   &=\int_B \int_B 1_A (e^{-(\cos\theta+\mi \sin\theta)s}(e^{-(\cos\theta+\mi \sin\theta)t}x+\sqrt{1-e^{-2t\cos \theta}}y)+\sqrt{1-e^{-2s\cos \theta}}z)\mu (\dif z) \mu(\dif y)\\
   &=\int_B \int_B 1_A (e^{-(\cos\theta+\mi \sin\theta)(s+t)}x+e^{-s\cos \theta}\sqrt{1-e^{-2t\cos \theta}}y+\sqrt{1-e^{-2s\cos \theta}}z)\mu (\dif z) \mu(\dif y)\\
   &\quad (\text{by the rotation invariant of the measure $\mu(\dif y)$})\\
   &=\int_B \int_B 1_A (e^{-(\cos\theta+\mi \sin\theta)(s+t)}x+y)\mu_{e^{-2s\cos \theta}(1-e^{-2t\cos \theta})}*\mu_{1-e^{-2s\cos \theta}}(\dif y)\\
   &=\int_B 1_A (e^{-(\cos\theta+\mi \sin\theta)(s+t)}x+\sqrt{1-e^{2(s+t)\cos \theta}}y)\mu(\dif y)\\
   &=P_{s+t}(x,\,A).
\end{align*}

The associated Markov process to $P_t(x,\,A)$ is called a \textbf{complex-valued Ornstein-Uhlenbeck process}.
Similar to the real case \cite[Proposotion 2.2]{shg}, it follows Kolmogorov's critertion and the rotation invariance of $\mu$ that the Ornstein-Uhlenbeck process is realized as a measure on $C([0,\,\infty)\to B)$.

The associated semigroup $\set{T_t,\,t\ge 0}$ is defined as follows: for a bounded Borel measurable function $F$,
\begin{equation}\label{sd_ou}
   T_tF(x)=\int_B\,F(e^{-(\cos \theta +\mi\sin\theta)t}x+\sqrt{1-e^{-2t\cos \theta}}y)\,\mu(\dif y).
\end{equation}
An argument similar to the one used in \cite[Proposition 2.3, 2.4]{shg} shows that
\begin{prop}
   $\mu$ is a unique invariant measure, i.e,
   \begin{equation*}
      \int_{B}P(t,\,A)\mu(\dif x)=\mu(A),\quad \forall A\in \mathcal{B}(B).
   \end{equation*}
And $\set{T_t,\,t\ge 0}$ is a strongly continuous contraction semigroup in $L^p(B,\mu)\,(p\ge 1)$.
\end{prop}

\section{The Ornstein-Uhlenbeck operator and the complex H-derivative}\label{ss3}
The generator of $\set{T_t}$ is called the Ornstein-Uhlenbeck operator, denoted by $L$. We will obtain a concrete expression of $L$ in this section.
Since there is a rotation term in the transition probabilities $P_t(x,A)$, to obtain a concrete expression of $L$, we need the complex H-derivative along a direction $\theta\in (-\frac{\pi}{2},\,\frac{\pi}{2})$.

\begin{dfn}
   A function $F:\,B\to \Cnum$ is {\it complex Gateaux differentiable} at $x\in B$ along the direction $\theta\in (-\frac{\pi}{2},\frac{\pi}{2})$ if there exist $\varphi_1,\,\varphi_2 \in B^*$ such that
   \begin{equation}
      \frac{\dif }{\dif t} F(x+e^{\mi \theta} ty)\big|_{t=0}=\innp{y,\,\varphi_1}+{\innp{\bar{y},\,\varphi_2}},\quad \forall y\in B.
   \end{equation}
$(\varphi_1,\varphi_2)$ is called a Gateaux derivative of $F$ at $x$ along the direction $\theta$, denoted by $\mathrm{G}_{\theta} F(x)$.
\end{dfn}
\begin{rem}
  Here we look $\varphi_1$ as a linear functional on $B$, and $\varphi_2$ a conjugate-linear functional. And we inherit the notation in \cite{shg} that
  \begin{equation*}
     _{B}\innp{x,\,\mathrm{G}_{\theta} F(x)}_{B^*}= \innp{x,\,\varphi_1}+{\innp{\bar{x},\,\varphi_2}}.
  \end{equation*}
\end{rem}
\begin{dfn}
   A function $F:\,B\to \Cnum$ is {\it complex H-differentiable} at $x\in B$ along the direction $\theta\in (-\frac{\pi}{2},\frac{\pi}{2})$ if there exist $h_1,\,h_2 \in H$ such that
   \begin{equation}
      \frac{\dif }{\dif t} F(x+e^{\mi \theta} th)\big|_{t=0}=\innp{h,\,h_1}+{\innp{h_2,\,h}},\quad \forall h\in H.
   \end{equation}
$(h_1,h_2)$ is called a {\it complex $H$-derivative} of $F$ at $x$ along the direction $\theta$, denoted by $\mathrm{D}_{\theta} F(x)$. When $\theta=0$, we denote $\mathrm{D}_{\theta} F(x)$ by $\mathrm{D}F(x)$ instead.
\end{dfn}

We can define higher order differentiability. For simplicity, we only present the 2-th case here.
\begin{dfn}\label{dn1}
$F$ is said to be 2-th $H$-differentiable along the direction $\theta$ if there exists a mapping $(\Phi_1,\,\Phi_2,\,\Phi_3,\,\Phi_4)$: $H\times H\to \Cnum^4$ such that $ \forall h_1,h_2\in H,$
\begin{align}\label{2-derivative}
   \frac{\partial ^2}{\partial t_1\partial t_2}F(x+e^{\mi \theta}t_1h_1+t_2h_2)\big|_{t_1=t_2=0}&=\sum_{j=1}^4 \Phi_j(h_1,h_2):=\Phi(h_1,h_2),
 \end{align}
 where $\Phi_1$ and $\bar{\Phi}_2$ are the bilinear forms \footnote{Here the bar is used for the conjugate instead of for the closure operator.}, and $\Phi_3$  and $\bar{\Phi}_4$ are the sesquilinear forms\footnote{The definition of sesquilinear is that the first argument is linear and the second one is conjugate-linear.}. $\Phi$ is called the 2-th H-derivative of $F$ at x along $\theta$, denoted by $\mathrm{D D}_{\theta} F(x)$.
\end{dfn}
\begin{dfn}
Let $\Phi$ be as in Definition~\ref{dn1}. $\Phi$ is said to be of trace class if the supremum
\begin{equation*}
   \sup \sum_{n=1}^{\infty} \sum_{i=1}^4\abs{\Phi_i(h_n,k_n)}
\end{equation*}
is finite, where $k_n$ and $h_n$ run over all c.o.n.s of $H$.
Furthermore, the trace of $\Phi$ is defined by
\begin{equation}\label{trac}
   \mathrm{tr}\, \Phi = \sum_{n=1}^{\infty} \Phi_1(h_n,\,\bar{h}_n)+  \Phi_2(h_n,\,\bar{h}_n)+ \Phi_3(h_n,\,h_n)+ \Phi_4(h_n,\,h_n).
\end{equation}
Here $\set{h_n}$ is a c.o.n.s of $H$, and this does not depend on a choice of c.o.n.s.
\end{dfn}
\begin{rem}
  An argument similar to the one used in \cite[p44]{rs} shows that there exist  bounded conjugate-linear operators $A_1,\,A_2$ such that $\Phi_1(h_1,h_2)=(h_1, A_1h_2)$ and $\Phi_2=(A_2h_2,h_1)$. 
\end{rem}

$\mathcal{S}$ stands for all functions $F:B\to \Cnum$ such that there exist $n\in \Nnum, \varphi_1,\dots, \varphi_n\in B^*,\,f\in C^{\infty}(\Cnum ^n)$ so that
\begin{equation}\label{sfc}
   F(x)=f(\innp{x,\varphi_1},\,\innp{x,\varphi_2},\dots,\innp{x,\varphi_n}).
\end{equation} Here we assume that $f$ with its derivatives has polynomial growth.
If $F\in \mathcal{S}$ , then the two derivative are given in the following forms.
Let $z_j=\innp{x,\varphi_j},\,j=1,\dots,n$ and denote
\begin{equation*}
   \partial_j f= \frac{\partial}{\partial z_j}f(z_1,\dots,z_n) ,\quad \bar{\partial}_j f=\frac{\partial}{\partial \bar{z}_j}f(z_1,\dots,z_n) .
\end{equation*}
If $\varphi\in B^*$, $c\varphi$ means that $(c\varphi)(x)=c\varphi(x)$. Then the G$\hat{a}$teaux derivative is
\begin{align}
   \mathrm{G}_{\theta} {F}(x)
   =\big( e^{\mi \theta}\sum_{j=1}^n \varphi_j \partial_j f,\quad e^{-\mi \theta} \sum_{j=1}^n  \varphi_j\bar{\partial}_j f\big),\\
  _{B}\innp{x,\,\mathrm{G_{\theta}} F(x)}_{B^*}
 =\sum_{j=1}^n[e^{\mi\theta}z_j\partial_j f +e^{-\mi \theta}\bar{z_j}\bar{\partial}_j f ].
\end{align}
The H-derivative is given by
\begin{align}
   \mathrm{D}_{\theta}F(x)&= \big( e^{\mi \theta}\sum_{j=1}^n \varphi_j \partial_j f,\quad e^{-\mi \theta} \sum_{j=1}^n  \varphi_j\bar{\partial}_j f\big),\\
   \mathrm{D}_{\theta}F(x)(h)&=\sum_{j=1}^n[e^{\mi\theta}\innp{h,\,\varphi_j }\partial_j f +e^{-\mi \theta}\innp{\varphi_j, \,h}\bar{\partial}_j f ],\label{1-deriva}
   \end{align}
where we adopt the convention that $B^*$ is the subspace of $H^*$.
(\ref{1-deriva})  implies that the 2-th H-derivative is given by
\begin{align*}
   \mathrm{D} \mathrm{D}_{\theta}^2 F(x)(h_1,h_2)&=\sum_{j,\,k=1}^n [ e^{\mi \theta} \innp{h_1,\varphi_j}(\innp{h_2,\varphi_k}\partial_{k}\partial_{j} f + \innp{\varphi_k,\, h_2})\bar{\partial}_{k}\partial_{j} f )\nonumber \\
   &+ e^{ -\mi\theta}\innp{\varphi_j,\, h_1}(\innp{h_2,\varphi_k}\partial_{k}\bar{\partial}_{j} f +\innp{\varphi_k,\, h_2}\bar{\partial}_{j}\bar{\partial}_{k} f )].
\end{align*}
If, in addition, $\set{\frac{\varphi_m}{\sqrt{2}},\,\frac{\bar{\varphi}_m}{\sqrt{2}}}$ is an
orthonormal system of $H^*$,
\begin{equation}
   \mathrm{tr}\,\mathrm{D D}_{\theta} F(x)=4\cos \theta \sum_{j=1}^n \partial_{j}\bar{\partial}_{j} f.
\end{equation}
\begin{prop}\label{pp3.3}
   For $F\in \mathcal{S}$,
   \begin{equation}\label{3-gene}
      L F(x)= \mathrm{tr}\,\mathrm{D D}_{\theta} F(x)- _{B}\innp{x,\,\mathrm{G_{\theta}} F(x)}_{B^*}.
   \end{equation}
\end{prop}
\begin{proof}
Suppose that $F\in \mathcal{S}$ is given by (\ref{sfc}). We may assume that $\set{\frac{\varphi_j}{\sqrt{2}},\,\frac{\bar{\varphi}_j}{\sqrt{2}}}$ is an orthonormal system of $H^*$. Thus
$\xi=(\innp{x,\varphi_1},\dots, \innp{x,\varphi})\in \Cnum^n$ has a $2n$-dimensional standard normal distribution and we have
\begin{equation*}
   T_t F(x)=\int_{\Cnum^n}\,f(e^{-(\cos \theta +\mi\sin\theta)t}\xi+\sqrt{1-e^{-2t\cos \theta}}\eta)\,(2\pi)^{-n}e^{-\abs{\eta}^2/2} \dif \eta.
\end{equation*}
When $t>0$,
\begin{align*}
   &\frac{\dif }{\dif t} T_t F(x)\\
   &=\frac{\dif }{\dif t} \int_{\Cnum^n}\,f(e^{-(\cos \theta +\mi\sin\theta)t}\xi+\sqrt{1-e^{-2t\cos \theta}}\eta)\,(2\pi)^{-n}e^{-\abs{\eta}^2/2} \dif \eta\\
   &=\sum_{j=1}^n\int_{\Cnum^n} (-\xi_j e^{\mi \theta}e^{-e^{\mi \theta}t}+ \frac{\eta_j\cos \theta e^{-2t\cos \theta}}{\sqrt{1-e^{-2t\cos \theta}}})\partial_j f(e^{-e^{\mi\theta}t}\xi+\sqrt{1-e^{-2t\cos \theta}}\eta) u(\dif\eta)\\
   &+ \sum_{j=1}^n\int_{\Cnum^n} (-\bar{\xi}_j e^{-\mi \theta}e^{-e^{-\mi \theta}t}+ \frac{\bar{\eta}_j\cos \theta e^{-2t\cos \theta}}{\sqrt{1-e^{-2t\cos \theta}}})\bar{\partial}_j f(e^{-e^{\mi\theta}t}\xi+\sqrt{1-e^{-2t\cos \theta}}\eta) u(\dif\eta)\\
   &=-e^{\mi \theta}e^{-e^{\mi \theta}t} \sum_{j=1}^n\xi_j \int_{\Cnum^n}\partial_j f u(\dif\eta) - e^{-\mi \theta}e^{-e^{-\mi \theta}t} \sum_{j=1}^n \bar{\xi}_j \int_{\Cnum^n} \bar{\partial}_j f u(\dif\eta) \\
   &+ {4\cos \theta e^{-2t\cos \theta}} \sum_{j=1}^n \partial_j \bar{\partial}_j f u(\dif\eta).
\end{align*}
The last equation is follows from  the ingegral by part of the complex creation operator(see Lemma2.3 of \cite{cl}). An argument similar to the one used in Proposition~2.7 of \cite{shg} shows that the convergence takes place in the topology of $L^p(B)$. Let $t\to 0$, we have
\begin{equation}
   L F(x)=-e^{\mi \theta}\sum^n_{j=1}\xi_j \partial_j f-e^{-\mi \theta}\sum^n_{j=1}\bar{\xi}_j\bar{ \partial}_j f +4\cos\theta \sum_{j=1}^n \partial_j \bar{\partial}_j f,
\end{equation}
which is exact (\ref{3-gene}).
\end{proof}
\section{It$\hat{\rm{o}}$-Wiener chaos decomposition, Eigenfunctions and the hypercontractivity}\label{ss4}
\begin{dfn}[Definition of the Hermite-Laguerre-It$\hat{\rm{o}}$ polynomials]\label{ijldf}
Let $m,n\in \Nnum$ and $z=x+\mi y$ with $x,y\in \Rnum$. We define the sequence on $\Cnum$
\begin{align}
  J_{0,0}(z)&=1,\nonumber\\
  J_{m,n}(z)&={2}^{m+n}(\partial^*)^m(\bar{\partial}^*)^n 1.\label{itldefn}
\end{align}
We call it the Hermite-Laguerre-It$\hat{\rm{o}}$ polynomial in the present paper.
\end{dfn}
One can show \cite{cl} that $\set{(m!n!2^{m+n})^{-\frac12}J_{m,n}(z):\,m,n\in \Nnum}$ is an orthonormal basis of $L^2_\Cnum(\nu)$ with $\dif\nu=\frac{1}{2\pi }e^{-\frac{x^2+y^2}{2}}\dif x\dif y$ and
\begin{equation}\label{2-eig}
    [e^{\mi \theta} z \frac{\partial}{\partial z}+ e^{-\mi \theta}\bar{z} \frac{\partial}{\partial \bar{z}}-4\cos \theta\frac{\partial^2}{\partial z\partial \bar{z}} ]J_{m,n}(z)=[(m+n)\cos \theta +\mi (m-n)\sin\theta]J_{m,n}(z).
\end{equation}

For a sequence $\mathbf{m}=\set{m_k}_{k=1}^{\infty}$, write $\abs{\mathbf{m}}=\sum\limits_{k}\,m_k$.
\begin{dfn}
   Take a complete orthonormal system in $\set{\varphi_k}\subseteq B^*$ in $H^*$ and fix it throughout the section. For two sequences $\mathbf{m}=\set{m_k}_{k=1}^{\infty},\,\mathbf{n}=\set{n_k}_{k=1}^{\infty}$ of nonnegative integrals with finite sum, define
  \begin{equation}\label{fourier}
    \mathbf{J}_{\mathbf{m},\mathbf{n}}(x):=\prod_{k} \frac{1}{\sqrt{2^{m_k+n_k}m_k!n_k!}}J_{m_k,n_k}(\innp{x,\,\varphi_k}).
  \end{equation}
  We name it the Fourier-Hermite-It$\hat{\rm o}$ polynomial. For two $m,n\in Z_{+}$, the closed subspace spanned by $\set{ \mathbf{J}_{\mathbf{m},\mathbf{n}}(x);\, \abs{\mathbf{m}}=m, \abs{\mathbf{n}}=n}$ in $L^2_{\Cnum}(B,\mu)$ is called the It$\hat{\rm o}$-Wiener chaos of degree of $(m,n)$ and is denoted by $\mathcal{H}_{m,n}$.
\end{dfn}

\begin{thm}
  For any fixed integer $m,n\ge 0$, the collection of functions
  \begin{equation}
    \set{\mathbf{J}_{\mathbf{m},\mathbf{n}};\abs{\mathbf{m}} =m,\,\abs{\mathbf{n}}=n }
  \end{equation}
  is an orthogonal basis for the space $\mathcal{H}_{m,n}$. And if $(m,n)$ vary then
  the collection of functions
  \begin{equation}
    \set{\mathbf{J}_{\mathbf{m},\mathbf{n}};\abs{\mathbf{m}} =m,\,\abs{\mathbf{n}}=n,\, m,n\ge 0 }
  \end{equation}
  is an orthogonal basis for the space $L^2_{\Cnum}(B,\mu)$. And $L^2_{\Cnum}(B,\mu)$ has the It$\hat{o}$-Wiener expansion in the following way:
  \begin{equation}
     L^2_{\Cnum}(B,\mu)=\bigoplus^{\infty}_{m=0}\bigoplus^{\infty}_{n=0}\mathcal{H}_{m,n}.
  \end{equation}
  The project from $L^2_{\Cnum}(B,\mu)$ to $\mathcal{H}_{m,n}$ is denoted by $\mathrm{J}_{m,n}$.
\end{thm}
The above theorem is well known, which is exact Example 3.32 of \cite[p31]{jan} which is from view of the Gaussian Hilbert spaces. The reader can also give an elementary proof using an argument similar to Theorem~9.5.4 and 9.5.7 of \cite{guo}. 
\begin{thm}\label{th4.4}
   Let $\mathbf{J}_{\mathbf{m},\mathbf{n}}(x)$ be a Fourier-Hermite-It$\hat{o}$ polynomial defined by (\ref{fourier}). Denote $m=\abs{\mathbf{m}},\,n=\abs{\mathbf{n}}$. Then
   \begin{align}
      L \mathbf{J}_{\mathbf{m},\mathbf{n}}(x)&=-[(m+n)\cos\theta  +\mi(m-n)\sin\theta]\mathbf{J}_{\mathbf{m},\mathbf{n}}(x),\label{lt}\\
      T_t \mathbf{J}_{\mathbf{m},\mathbf{n}}(x)&= e^{-[(m+n)\cos\theta  +\mi(m-n)\sin\theta]t}\mathbf{J}_{\mathbf{m},\mathbf{n}}(x).\label{tt}
   \end{align}
\end{thm}
\begin{proof}
Proposition~\ref{pp3.3} and (\ref{2-eig}) imply (\ref{lt}) directly. (\ref{tt}) follows from (\ref{lt}) and semigroup equation (or say: Kolmogorov's equaiton).
\end{proof}
In fact, (\ref{tt}) is an alternative procedure for introducing the OU semigroup.
Similar to the symmetric OU semigroup(see \cite[p54]{Nua}), we define a nonsymmetric OU semigroup:
\begin{dfn}
The nonsymmetric OU semigroup is the one-parameter semigroup $\set{T_t,\,t\ge 0}$ of contraction operators on $L^2_{\Cnum}(B) $ defined by
\begin{equation}\label{ou}
   T_t F(x)=\sum_{m=0}^{\infty}\sum_{n=0}^{\infty} e^{-[(m+n)\cos\theta  +\mi(m-n)\sin\theta]t}\mathrm{J}_{m,n}F
\end{equation}
for any $F\in L^2_{\Cnum}(B)$.
\end{dfn}

Finially, similar to the proof of \cite[Theorem 2.11]{shg}, we have the hypercontractivity of the OU semigroup.
\begin{prop}
  For the fixed $t\ge 0$ and $p>1$, set $q(t)=e^{2t}(p-1)+1$. Then
  \begin{equation}
     \norm{T_t F}_{q(t)}\le \norm{F}_{p},\quad \forall F\in L^p(B,\mu).
  \end{equation}
\end{prop}


\begin{proof}
Since $\mathcal{S}$ is dense in $ L^p(B,\mu)$, it is enough to show this when $B=\Cnum^n$. It is enough to show that for any $0<a\le f,\,g\le b$ where $f,\,g$ are Borel functions on $\Cnum^n$, the following inequality holds \cite[Theorem 2.11]{shg}.
\begin{equation}\label{pphy}
   \int_{\Cnum^n} T_t f(\xi)g(\xi)(2\pi)^{-n}e^{-\abs{\xi}^2/2}\dif \xi \le \norm{f}_{p}\norm{g}_{q(t)'}.
\end{equation}

Let $\zeta_t=(\zeta_t^{(1)},\,\zeta_t^{(2)},\dots,\zeta_t^{(n)})',\,0\le t\le 1$ be an n-dimensional standard complex Brownian motion. $\bar{\zeta}_t$ is the complex conjugate.
Let $\tilde{\zeta}_t$ be an independent copy of $\zeta_t$. For a given $0<\lambda<1$ and $a\in \Rnum$, define
\begin{equation}
  \hat{\zeta}_t=\lambda e^{\mi a}\zeta_t+\sqrt{1-\lambda^2}\tilde{\zeta}_t.
\end{equation}
Clearly, $\hat{\zeta}_t$ is still a standard complex Browian motion. Set $\mathcal{F}^\zeta_t=\sigma(\zeta_s;\,0\le s\le t),\, \mathcal{F}^{\hat{\zeta}}_t=\sigma(\hat{\zeta}_s;\,0\le s\le t)$. Define martingales
\begin{align*}
   M_t=E[f^p(\hat{\zeta}_1)|\, \mathcal{F}^{\hat{\zeta}}_t],\quad 0\le t\le 1,\\
   N_t=E[g^{q'}({\zeta}_1)|\, \mathcal{F}^{{\zeta}}_t],\quad 0\le t\le 1,
\end{align*}
where $q'$ is the conjugate number of $$q=1+(p-1)/\lambda^2.$$
It follows from the martingale (on filtrations induced by the complex Brownian motion) representation theorem that
\begin{align*}
   M_t=M_0+ \int_0^t\,\theta_s\dif \hat{\zeta}_s+ \int_0^t \vartheta_s\dif \bar{\hat{\zeta}}_s,\quad
   N_t=N_0+\int_0^t\,\phi_s\dif \zeta_s+ \int_0^t \varphi_s\dif \bar{\zeta}_s.
\end{align*}
Since $ M_t,\,N_t\in \Rnum$, $\vartheta_s=\bar{\theta}_s$ and $\varphi_s=\bar{\phi}_s$.
It follows from the It$\hat{\rm o}$'s table that $\dif M_t\dif M_t=4\abs{\theta_t}^2 \dif t,\,\dif N_t\dif N_t=4\abs{\phi_t}^2\dif t$ and $\dif M_t \dif N_t=2\lambda(e^{\mi a}\theta_t\varphi_t+e^{-\mi a}\vartheta_t\phi_t)\dif t$. By It$\hat{\rm o}$'s formula, we have
\begin{align*}
   \dif (M_t^{1/p}N_t^{1/q'})&=\frac{1}{p}M_t^{1/p-1}N_t^{1/q'}\dif M_t+\frac{1}{q'}M_t^{1/p}N_t^{1/q'-1}\dif N_t\\
   &+\frac12 \frac{1}{p}( \frac{1}{p}-1)M_t^{1/p-2}N_t^{1/q'}\dif M_t\dif M_t\\
   &+\frac{1}{p}\frac{1}{q'}M_t^{1/p-1}N_t^{1/q'-1}\dif M_t\dif N_t\\
   &+ \frac12 \frac{1}{q'}( \frac{1}{q'}-1)M_t^{1/p}N_t^{1/q'-2}\dif N_t\dif N_t
\end{align*}
Note that $\sqrt{(p-1)(q'-1)}=\lambda$, therefore,
\begin{align*}
   &E(M_t^{1/p}N_t^{1/q'})-E(M_0^{1/p}N_0^{1/q'})\\
   &= -2 E[\int_0^t M_t^{1/p-2}N_t^{1/q'-2}[\frac{1}{p}( 1-\frac{1}{p})N_t^2\abs{\theta_t}^2-2\frac{1}{p}\frac{1}{q'}\lambda M_tN_t\Re(e^{\mi a}\theta_t\varphi_t)\\
   &\quad +\frac{1}{q'}( 1-\frac{1}{q'})M_t^2\abs{\phi_t}^2 ]\dif t]\\
   &=-2 E\big[\int_0^t M_t^{1/p-2}N_t^{1/q'-2}\abs{ \frac{\sqrt{p-1}}{p}N_t \theta_t- \frac{\sqrt{q'-1}}{q'}e^{\mi a}M_t\phi_t}^2\dif t\big]\\
   &\le 0.
\end{align*}
Let $t=1$ in the above inequality displayed, we have
\begin{equation*}
   E(f(\hat{\zeta}_1)g(\zeta_1))\le E[f^p(\hat{\zeta}_1)]^{1/p}E[g^{q'}(\zeta_1)]^{1/{q'}}.
\end{equation*}

From the definition of $\hat{\zeta}$ and letting $\lambda=e^{-t\cos\theta},\,a=\sin\theta$,  the above inequality displayed is exact (\ref{pphy}). This ends the proof.
\end{proof}

An argument similar to the one used in\cite[Proposition 2.14, 2.15]{shg} shows the following boundedness of operator in $L^p(B,\mu)\,(p>1)$.
\begin{cor}\label{cor4}
$\mathcal{H}_{m,n}$, the It$\hat{\rm o}$-Wiener chaos of degree of $(m,n)$, is a closed subspace in $L^p(B,\mu)\,(p>1)$ and its norms $\norm{\cdot}_p$ in $L^p(B,\mu)\,(p>1)$ are equivalent to each other. In addition, the project operator $\mathrm{J}_{m,n}$ is bounded in $L^p(B,\mu)\,(p>1)$ and satisfies that
\begin{align}
   \mathrm{J}_{m,n}\mathrm{J}_{i,j}=\mathrm{J}_{i,j}\mathrm{J}_{m,n}=\delta_{m,i}\delta_{n,j}\mathrm{J}_{m,n}\\
   T_t\mathrm{J}_{m,n}=\mathrm{J}_{m,n}T_t=e^{-(m+n)t\cos \theta-\mi (m-n)t\sin\theta}\mathrm{J}_{m,n}.
\end{align}
\end{cor}
\vskip 0.2cm {\small {\bf  Acknowledgements}}\   This work was
supported by  NSFC( No.11101137.) 


\end{document}